\documentclass[12pt]{article}
\usepackage{amssymb}

\usepackage{amsthm}
\usepackage{amsmath}

\newtheorem{defi}{\bf Definition}[section]
\newtheorem{theo}[defi]{\bf Theorem}

\newtheorem{coro}[defi]{\bf Corollary}
\newtheorem{pro}[defi]{\bf Proposition}
\newtheorem{lem}[defi]{\bf Lemma}

\title{The derived superalgebra of skew elements of a semiprime superalgebra with superinvolution}

\author{Jes\'us Laliena  }

\date{\quad}

\begin{document}
\maketitle\vspace{-1.5cm}

\begin{abstract}

In this paper we investigate the Lie structure of the derived Lie superalgebra $[K, K]$, with $K$ the set of skew elements of a
semiprime associative superalgebra $A$ with superinvolution. We show that if $U$ is a Lie ideal
of $[K, K]$, then either there exists an ideal $J$ of $A$ such that the Lie ideal   $[J\cap K,K]$ is nonzero and contained in $U$, or  $A$ is a subdirect sum of $A'$, $A''$, where the image of $U$ in $A'$
is central, and $A''$ is a subdirect product of orders in simple superalgebras, each at most
16-dimensional over its center.

\end{abstract}

{\parindent= 4em \small  \sl Keywords: associative superalgebras, semiprime superalgebras, superin-}

{\parindent=10em \small  \sl  volutions, skewsymmetric elements, Lie structure.}

{\parindent=4em \small \bf Classification MSC 2010: 16W55, 17A70, 17C70}

\section{Introduction.}

\bigskip

Let $A$ be an algebra over $\phi$, an associative commutative unital ring of scalars with $1/2 \in \phi$. $A$ is said to be a superalgebra if it is a ${\bf Z}_2$-graded algebra, that is, $A= A_0 + A_1$, with $A_i A_j \subseteq A_{i+j}, i, j \in {\bf Z}_2$. $A_0$ is said to be the even part  and $A_1$ is said to be the odd part. Elements in $A_0$ and $A_1$ are said to be homogeneous elements.

A Lie superalgebra is a superalgebra with an operation $[\ , \ ]$  satisfying the following axioms for every $a, b, c$ homogeneous elements in $A$ (where $\bar a$ denotes the degree of $a$, that is $a\in A_{\bar a}$)

\begin{align*} 
[a, b]=& -(-1)^{\bar a \bar b} [b,a] \\
[a, [b, c]] =& [[a,b],c] + (-1) ^{\bar a \bar b} [b, [a, c]]
\end{align*}

Superalgebras have proved to be very useful in mathematics, and, in particular,  in algebra. For example in the theory of varieties of algebras, in questions concerning the structure of $T-$ideals, their nilpotence or solvability (\cite {Ke}, \cite {Z}, \cite{S-Z}); and also to construct some counterexamples, for instance, of solvable but not nilpotent Jordan, alternative and $(-1, -1)$-algebras, or to construct prime algebras with nonzero absolute zero divisors (\cite {S}).

In the last two decades, the different kinds of superalgebras have been profusely investigated, and also the relationships among them. In this paper we are interested in study some relationships among associative and Lie superalgebras. More specifically  we are interested in the description of the Lie structure of the derived superalgebra $[K,K]$, with $K$ the set of skewsymmetric elements of a semiprime superalgebra with superinvolution.

An associative superalgebra is just a superalgebra that is associative as an ordinary algebra.

 It is known that, if we take an associative superalgebra, $A$, and we change the product in $A$ by the superbracket product $[a, b]= ab- (-1)^{\bar a \bar b} ba$, where $\bar a, \bar b$ denotes the degree of $a$ and $b$, homogeneous elements in $A=A_0+A_1$, we obtain a Lie superalgebra, denoted by $A^-$. Also if $A$ is an associative superalgebra and has a  superinvolution, that is, a graded linear map $* : A \longrightarrow A$ such that $a^{* *}=a$ and $(ab)^*= (-1)^{\bar a \bar b} b^* a^*$, for $a, b\in A$ homogeneous elements, the set of skewsymmetric elements, $K=\{ x \in A : x^*= -  x\}$, is a subalgebra of the Lie superalgebra $A^-$. In fact, in the classification of the finite dimensional simple Lie superalgebras given by V. Kac  in \cite{Ka}, several types are of this kind.

This important fact made that, in \cite {Go-S}, C. G\'omez-Ambrosi and I. Shestakov investigated  the Lie structure of the set of skew elements, $K$, and also of $[K, K]$, of a simple associative superalgebra with superinvolution over a field of characteristic not 2. More specifically they described the ideals of these Lie superalgebras $K$ and $[K, K]$, also called Lie ideals of $K$ and $[K,K]$. Those results were extended in \cite {Go-L-S} to prime associative superalgebras with superinvolution for the Lie superalgebras $K$ and $[K, K]$, and later, in \cite {L-S} to semiprime superalgebras with superinvolution, but only for the Lie superalgebra $K$.

We notice that  the Lie structure of prime associative superalgebras and simple associative superalgebras without superinvolution was investigated by F. Montaner (\cite {M}) and S. Montgomery (\cite {Mo}). 

In the non graded case, there is a parallel situation for associative algebras with involution and Lie algebras. This fact was first studied by I. N. Herstein (\cite {H1}, \cite {H2}) and W. E. Baxter (\cite {B}), and after by several authors: T. E. Erickson (\cite {Er}), C. Lanski (\cite {La}, W. S. Martindale III and C. R. Miers (\cite {Ma}), \dots

For a complete introduction to the basic definitions and examples of superalgebras,
superinvolutions and prime and semiprime superalgebras, we refer the reader to \cite {Go-S} and \cite
{M}.

Throughout the paper, unless otherwise stated,  $A$ will denote a nontrivial semi-prime associative
superalgebra with superinvolution * over an associative  commutative unital ring $\phi$ of scalars with ${1\over
2}\in \phi$. By a nontrivial superalgebra we understand a superalgebra with nonzero odd part. $Z$ will
denote the even part of the center of $A$, $H$ the Jordan superalgebra of symmetric elements of $A$,
and $K$ the Lie superalgebra of skew elements of $A$. If $P$ is a subset of $A$, we will denote by
$P_H=P\cap H$ and $P_K= P\cap K$. The following containments are straightforward to check, and they
will be used throughout without explicit mention:
$[K,K]\subseteq K, \quad [K,H]\subseteq H, \quad [H,H]\subseteq K, \quad H\circ H \subseteq H, \quad
H\circ K\subseteq K$ and
$K\circ K\subseteq H$.

We recall that a superinvolution * is said to be of the first kind if $Z_H=Z$, and 
of the second kind if $Z_H\not= Z$.

If $Z\not= 0$, one can consider the localization $Z^{-1}A=\{z^{-1}a : 0\not= z \in Z, a\in A\}$. If
$A$ is prime, then  $Z^{-1}A$ is a central prime associative superalgebra over the field
$Z^{-1}Z$.  We call this superalgebra  the central closure of $A$. We also say that $A$ is a
central order in $Z^{-1}A$. This terminology is not the standard one, for which the definition involves
 the extended centroid. 

Let $A$ be a prime superalgebra, and let  $V=Z_H-\{0\}$ be the subset of regular symmetric elements. Note that if $Z\not= 0$, $Z_H\not=0$. Also  $Z^{-1}A=V^{-1}A$, since for all  $0\not=z \in Z, a\in A$ we have $z^{-1}a=(zz^*)^{-1}(z^*a)$. It will be more convenient for us, in order to extend the superinvolution in a natural way, to work with $V$ rather than with $Z$.  We may consider $V^{-1}A$ as a superalgebra over the field $V^{-1}Z_H$. Then the superinvolution on $A$ is extended to a superinvolution  of the same kind on $V^{-1}A$ over  $V^{-1}Z_H$
via $(v^{-1}a)^*=v^{-1}a^*$. It is then easy to check that $H(V^{-1}A,*)=V^{-1}H$ and
$K(V^{-1}A,*)=V^{-1}K$. Moreover, $Z(V^{-1}A)_0=V^{-1}Z$ and $V^{-1}Z\cap V^{-1}H=V^{-1}Z_H$. We will say that the superalgebra $V^{-1}A$ over the field $V^{-1}Z_H$ is the *-central closure of $A$.

 We notice that in every semiprime superalgebra $A$,  the intersection of all the prime ideals $P$
of $A$ is zero. Consequently $A$ is a subdirect product of its prime images. If each prime image of
$A$ is a central order in a simple superalgebra at most $n^2$ dimensional over its center, we  say
that $A$ verifies $S(n)$.

In this paper, we prove that if $K$ is the Lie superalgebra of skew elements of a semiprime
associative superalgebra with superinvolution, $A$, and $U$ is a Lie ideal of $[K,K]$, that is, $U$ is a $\phi$-submodule of $[K,K]$ such that $[U,[K,K]]\subseteq U$,  then one of the following alternatives must hold: either $U$ must contain a nonzero Lie ideal $[J\cap K,K]$, for $J$ an ideal of $A$, or  $A$ is a subdirect sum of $A^\prime$, $A^{\prime \prime}$, where the image of $U$ in $A^{\prime \prime}$ is central and $A^\prime $ satisfies $S(4)$.

The following results are instrumental for the paper:

\begin{lem} (\cite {H3}, Theorem 1)
Let $A$ be  a semiprime algebra  and let $L$ be a Lie ideal of $A$. If  $[a,[a,L]]=0$, then $[a,L]=0$.
\end{lem}

\begin{lem} (\cite {M}, Lemmata 1.2, 1.3)
If $A=A_0 \oplus A_1 $ is a semiprime superalgebra, then  $A_0$ is a semiprime algebra. Moreover, if $A$ is prime, then 
either $A$ is prime or $A_0$ is prime (as algebras).
\end{lem}

\begin{lem} (\cite {M}, Lemma 1.8)
Let $A=A_0 \oplus A_1$ be a prime superalgebra. Then

\begin{enumerate}
\item[{\rm (i)}] If $x_1 \in A_1 $ centralizes a nonzero ideal $I$ of $A_0$, then
$x_1 \in Z(A)$.

\item[{\rm (ii)}] If $x_1 ^2$ belongs to the center of a nonzero ideal $I$ of $A_0$, then
$x_1 ^2 \in Z(A)$.
\end{enumerate}
\end{lem}

\bigskip

\begin{lem} (\cite {Go-L-S}, Corollary 2)
Let $A$ be a semiprime superalgebra and $L$ a Lie ideal of $A$. Then either $[L,L]=0$, or $L$ is dense
in $A$.
\end{lem}

\begin{lem}(\cite {Go-L-S}, Theorem 2.1)
Let $A$ be a prime nontrivial associative superalgebra. If $L$ is a Lie ideal of $A$,  then either
$L\subseteq Z$ or $L$ is dense in $A$, except if $A$ is a central order in a 4-dimensional Clifford
superalgebra.
\end{lem}

We remark that the bracket product in Lemma 1.1 is the usual one: $[a,b]=ab-ba$, but the bracket
product in Lemmata 1.4,1.5 is the superbracket $[x_i,y_j]_s=x_iy_j-(-1)^{ij}y_jx_i$ for $x_i\in
A_i, y_j\in A_j$ homogenous elements. In fact, the superbracket product coincides with the usual
bracket if one of the arguments belongs to the even part of $A$. In the following, to simplify the
notation, we will denote both in the usual way $[\ , \ ]$ but we will understand that it is the
superbracket if we are in a superalgebra.

\bigskip

During the paper we will use very often the following identities in a superalgebra $A$, for $a, b, c$ homogeneous elements in $A$:

\bigskip

\begin{gather}
[a, bc]= [a,b] c + (-1)^{\bar a \bar b} b[a,c], \\
[ab,c] = a[b,c]+ (-1)^{\bar b \bar c} [a,c]b\\
[a,b\circ c]= [a,b]\circ c + (-1)^{\bar a \bar b} b\circ [a,c]\\
[a\circ b, c] = a \circ [b,c] + (-1)^{\bar b \bar c} [a,c]\circ b
  \end{gather}

\section {Lie structure of $[K,K]$.}

\bigskip
 
Let $A$ be an associative superalgebra and $M,S$ be $\Phi$-submodules of $A$. Define $(M:S)=\{a\in A :
aS\subseteq M\}$, and denote by $\overline M$ the subalgebra of $A$ generated by $M$. We will say that $M$ is dense in $A$ if $\overline M$ contains a nonzero ideal of $A$.

Also we define the following multiplication on $A$: $u\circ v= uv + (-1)^{\bar u \bar v} vu$.

 Let $U$ be a Lie ideal of $[K,K]$. We recall  (see Lemma 4.1 in \cite {Go-S}) that $K^2$ is a Lie ideal of $A$. So $\overline {K^2}$ is also a Lie ideal or $A$, because for every $k, l $ homogeneous elements in $K^2$ and for every $a$ homogenous element in $A$ we have $[k l, a] = k [l, a ] + (-1)^{\bar l \bar a} [k, a] l\in \overline {K^2}$.
 
\begin{lem}
If $A$ is semiprime, then either $U$ is dense in $A$ or $[u\circ v,w]=0$
for every homogeneous elements $u,v\in [U,U], w\in U$.
\end{lem}

\begin{proof}[Proof:]
We present the proof of this in six steps. Let $u,v \in [U,U], w\in U$.

\medskip

1. $[u\circ v, w] \in (\overline U:A)$ . We have 
 $$[u\circ v, k]= u\circ [v,k] + (-1)^{{\bar k}{\bar v}} [u,k]\circ v\in \overline{U}$$
{\noindent for every homogeneous elements
$u,v \in [U,U]$ and $k\in K$, because }
$$[[U,U], K] \subseteq [U,[U,K]]\subseteq [U,[K,K]]\subseteq U.$$
{\noindent  And also for every homogeneous elements $u,v\in [U,U]$ and $h\in H$ we get}
$$[u\circ v,h]=[u,v\circ h]+(-1)^{\bar u \bar v}[v,u\circ h]\in  U,$$
{\noindent because $K\circ H\subseteq K$. Since $A= H\oplus K$ it
follows that  $[u\circ v, A]\subseteq \overline U$ for every
homogeneous elements $u,v\in [U,U]$. But for every homogeneous elements $a\in A$, $w\in U$}
$$ [u\circ v, wa]=[u\circ v, w]a + (-1)^{({\bar u}+{\bar v}){\bar w}} w[u\circ v,a]$$
{\noindent and so $[u\circ v,w] A\subseteq \bar U$, that is, $[u\circ v,w]\in
 (\bar U:A)$.
 
 \medskip
 
 2.  $[u\circ v, A]\subseteq \  \overline {K^2}, \overline {[K,K]}$ and $[u\circ v, w]\in (\ \overline{K^2}:A),  (\ \overline {[K,K]}:A)$.  We notice that from the above equations we can also deduce that $[u\circ v,A] \subseteq 
\overline {[K,K]}$ and  that $[u\circ v,w]\in(\overline {[K,K]}:A)$}.

\medskip

3. $A [u\circ v, w] A \subseteq \overline{ K^2}$.  We claim that $A[u\circ v, w] \subseteq (\overline { K^2}: A)$. Let  $a, b\in A$ homogenous elements, then
$$  a[u\circ v, w]b=[a,[u\circ v, w]b]+(-1)^{ (\bar u + \bar v + \bar w) \bar a + \bar b \bar a} [u\circ v, w]ba\subseteq \overline {K^2},$$
{\noindent because of step 2 and because $\overline {K^2}$ is a Lie ideal of $A$.}

\medskip

 4.  $\overline K .  (\overline {[K,K]}: A) \subseteq (\overline {[K,K]}: A)$.   Let $k\in K, x\in (\overline {[K,K]}:A), a \in A$ homogeneous elements, then
$$(kx)a=[k,xa]+(-1)^{(\bar x +\bar a)\bar k}(xa)k\in \overline{[K,K]},$$
{\noindent because $x\in (\overline { [K,K]}:A)$ and because  if $l, m \in [K,K]$ are homogeneous elements then from  (1)  }
$$[k,lm] = [k,l]m + (-1)^{\bar k \bar l} l[k,m] \in \overline {[K,K]}.$$

5.  $\overline {[K,K]}. (\bar U:A) \subseteq (\bar U:A)$. It is the same proof as in step 4. Let $k\in [K,K], x\in (\bar U:A), a \in A$ homogeneous elements, then
$$(kx)a=[k,xa]+(-1)^{(\bar x +\bar a)\bar k}(xa)k\in \bar U,$$
{\noindent because $x\in (\bar U:A)$ and because if $l, m \in U$ are homogeneous elements then  }
$$[k,lm] = [k,l]m + (-1)^{\bar k \bar l} l[k,m] \in \overline {U},$$
{\noindent since $U$ is a Lie ideal of $[K,K]$.}

\medskip

6. $A [u\circ v, w] A [u \circ v, w] A [u \circ v, w] A \subseteq \bar U$. From steps 1--5 we deduce that
$$A[u\circ v, w]A[u\circ v, w]A[u\circ v, w]A\subseteq \overline {K^2} (\overline {[K,K]}: A) A (\bar U: A) A \subseteq \overline {[K,K]} ({\overline U}:A) A\subseteq \overline U .$$
{\noindent So, if $[u\circ v, w] \not= 0$, since $A$ is semiprime,  $0\not= J= A[u\circ v, w]A[u\circ v, w]A[u\circ v,w]A \subseteq \bar U$, and then $U$ is dense in $A$.}

\end{proof}

\bigskip

We note that the ideal contained in $\bar U$ in the above Lemma, $J=A[u\circ v,w]A[u\circ v,w]A[u\circ v, w]A$, is also a $*$-ideal, that is, $J^*\subseteq J$.

\begin{theo}
Let $A$ be a semiprime superalgebra with superinvolution, then either $K$ is dense or $A$ satisfies $S(2)$.
\end{theo}

\begin{proof}[Proof:]
Consider the Lie ideal of $A$,  $K^2$. From Lemma 1.4, either $K^2$ is dense in $A$, or $[K^2,K^2]=0$. In the first case,  $K$ is dense in $A$, clearly. In the second case, by Theorem 1.1 in \cite{L-Sa}, $A$ satisfies $S(2)$.
\end{proof}

\begin{lem}
Let $A$ be semiprime, and let $U$ be a Lie ideal of $[K,K]$ such that $[u\circ v, w]=0$ for every $u,v \in [U,U], w\in U$.  Then

\begin{enumerate}

\item[{\rm (i)}] $u\circ v\in Z$ for every $u,v \in [U,U]_i$.

\item[{\rm (ii)}] $(u\circ v)^2=0$ for every $u\in [U,U]_0, v\in [U,U]_1$.

\item[{\rm (iii)}] $u\circ v=0$ for every $u,v\in [U,U]_1$.

\end{enumerate}

\end {lem}

\begin{proof}[Proof:]
From step 1  and its proof in Lemma 2.1,  we know that $[u\circ v, h] \in U$ and $[u\circ v, k] \in \bar U$ for every homogeneous elements $u,v \in [U,U], h\in H, k \in K$. Therefore $[u\circ v, a]\in \overline U$ for every $a\in A$.  So, from (1),  $[u\circ v, [u\circ v, a]]=0$. Now,  if $u\circ v$ is even, we obtain from Lemma 1.1   that $u\circ v\in Z$ and we have (i). And if $u\circ v$ is odd, then, from (4),
 $$[u\circ v, u\circ v]= (-1)^{\bar u \bar u+ \bar u \bar v} u\circ [u\circ v, v]+ [u\circ v, u] \circ v=0,$$
 {\noindent  that is, $(u\circ v)^2=0$, and we have (ii). 
 
 Now, suppose that $\gamma = u\circ v$ with $u, v \in [U,U]_1$.  Then}
\begin{eqnarray*}
\gamma (u^2 \circ v)&=& u^2 \circ \gamma v = \frac {1}{2} (u^2 \circ ((u\circ v) \circ v)= -\frac {1}{2} ( u^2 \circ [ v^2, u ])\\&=& -\frac {1}{2}( [u^2 \circ v^2, u])  = -\frac {1}{2}( [[u,u]\circ [v,v], u])=0,
\end{eqnarray*}
{\noindent  because $\gamma \in Z$, because of the hypothesis and from (3). A similar argument shows that $\gamma (v^2\circ u)=0$.  Notice that $0=[u\circ v, u]=[uv-vu,u]=uvu-vu^2-u^2v+uvu$, and so $2uvu= u^2\circ v$. Therefore $\gamma (uvu)=0$. And since we can also prove that $2vuv= v^2\circ u$,  it  is  deduced that  $\gamma (vuv) =0$. Now we observe that }
$$2\gamma u^3= \gamma u \circ u^2=\frac {1}{2} ((u\circ v)\circ u)\circ u^2= \frac {1}{2} [u^2,v]\circ u^2= \frac {1}{2}[u^2, v\circ u^2]=0$$
{\noindent because of the hypothesis and from (4). And the same $\gamma v^3=0$. Notice that}
\begin{eqnarray*}
\gamma ^2 &=& (u\circ v) (u\circ v)= (\gamma u \circ v) = 1/2 (\gamma \circ u) \circ v = 1/2 ((u\circ v) \circ u) \circ v \\&=& 1/2 ([u^2,v] \circ v)= 1/2 (-[u^2\circ v, v] + u^2 \circ [v,v])= 1/2 (u^2 \circ  v^2),
\end{eqnarray*}
{\noindent and so finally}
\begin{eqnarray*} \gamma ^4 &=& \gamma \gamma \gamma ^2= {1\over 2}\gamma (u\circ v) (u^2\circ v^2) ={1\over 2}\gamma (uv-vu) (u^2 v^2+v^2u^2)\\&= &{1\over 2}\gamma (uvu^2v^2+uv^3u^2-vu^3v^2-vuv^2u^2)=0, 
\end{eqnarray*}

{\noindent because  $\gamma uvu = \gamma vuv = \gamma u^3 = \gamma v^3=0$. So, since $A$ is semiprime, we obtain that $\gamma =0$ and we get (iii)}
\end{proof}

\bigskip

In the following two sections we deal with the second case of Lemma 2.1, that is, when $[u\circ v,w]=0$ for every $u,v\in [U,U], w\in U$, and we will study the prime images of $A$. If $P$ is a prime ideal of $A$ we have two posible situations: either $P^*\not= P$ or $P^*=P$.

\bigskip

\section {Prime images of Lie ideals  when $P^*\not= P$.}

\bigskip

Let $P$ be a prime ideal of $A$. We will suppose first that $P^*\not= P$. In this case $(P^*+P)/P$ is
a nonzero proper ideal of $A/P$ and we claim that $(P^*+P)/P\subseteq (K+P)/P$. Indeed, if
$y\in P^*$ then $y+P=(y-y^*)+ y^*+P\in (K+P)/P$. Also if $U$ is a Lie ideal of
$[K,K]$ we have that $(U+P)/P$ is a $\phi$-submodule of $A/P$ and satisfies
$$[(U+P)/P,[(P^*+P)/P, (P^*+P)/P]]\subseteq ([U,[K,K]]+P)/P\subseteq (U+P)/P.$$

Of course if $u\circ v\in Z$ for every $u,v \in [U,U]_0$,  $u \circ v =0$ for every $u,v\in
[U,U]_1$,  and $(u\circ v)^2 =0$ for every $u\in [U,U]_0, v\in [U,U]_1$, then the same property is satisfied in $A/P$, that is, $(u+P) \circ (v+P) \in Z_0 (A/P)$ for
every $u+P, v+P \in ([U,U]_0+P)/P$, $(u+P)\circ (v+P)=0$ for every $u+P, v+P\in  ([U,U]_1+P)/P$, and $((u+P)\circ (v+P))^2=0$ for every $u+P \in ([U,U]_0+P)/P, v+P \in ([U,U]_1 + P)/P$.  Let us analyze this situation.  We notice that the assumption that $A/P$ has a superinvolution is not required. We state first some useful Lemmata.

\begin{lem}
Let $A$ be a prime superalgebra, $I$ a nonzero ideal of $A$,  then either $[I,I]$ is dense in $A$,  or $A$ is a central order in a 4-dimensional  Clifford superalgebra, or $A$ is commutative.
\end{lem}

\begin{proof}[Proof:]

We notice that  $[I,I]$ is a Lie ideal of $A$, and from Lemma 1.5 it follows that either $[I,I]$ is dense in $A$, or $A$ is a central order in a 4-dimensional Clifford superalgebra, or $[I, I]\subseteq Z$. Suppose that $[I,I]\subseteq Z$, then $[I_0,I_1]=0$. But then, from Lemma 1.3 (i) we deduce that $I_1\subseteq Z_1(A)$. We observe that $I_1 \not= 0$ because  if $I=I_0$  then $I. (A_1 + A_1^2)=0$, a contradiction  with the primeness of $A$. Therefore $I=I_0+I_1 $ with $I_1 \not= 0$, and this is satisfied for every nonzero ideal of $A$.  Let $J=I_0I_1+I_1^2$. Since $I_1\subseteq Z_1(A)$, $J$ is anideal of $A$. Also $J\not=0$, because if $J=0$, then $0\not= I^2=I_0^2$ by primeness, and $(I^2)_1=0$, a contradiction. Since $I_1 \subseteq Z_1(A)$,  we get $[x,J]=0$  for every $x\in I_0$, because of (1). Therefore for every $a\in A$ and $y\in J$ we have $(xa)y=(ay)x=(ax)y$, that is, $(xa-ax) J=0$, and since $A$ is prime we deduce that $xa = ax$ for every $a\in A$, so $I_0\subseteq Z$, and then $I\subseteq Z(A)$. Now is easy to prove that $A$ is commutative. For every homogeneous elements $a, b \in A$ and $y\in I$ it follows that
$$ (ab)y= (by)a=(yb)a=(ba)y,$$
{\noindent and by the primeness of $A$,  $ab= ba$ for every homogeneous elements $a, b \in A$.}

\end {proof}

\begin{lem}
Let $A$ be a prime superalgebra,  $L$  a Lie ideal of $A$ such that $L$ is dense in $A$, and $v\in A_i$ such that   $vLv=0$, then $v=0$.
\end{lem}

\begin{proof}[Proof:]
Let $u\in L_i$ and $a\in A$. Then $v[u,a]v=0$. Considering now $v[u,u^\prime va]v$ with $u^\prime \in L$, homogeneous,  we have $vuu^\prime vav=0$. Therefore $vuu^\prime vA$ is a right ideal with square zero, that is a contradiction with the primeness of $A$, so $vuu^\prime v=0$. In the same way considering $v[u, u^\prime u^{\prime \prime}va]v$ we obtain that $vuu^\prime u^{\prime \prime}v=0$. So, if $J$ is a nonzero ideal such that $J\subseteq \bar L$, we deduce that $vJv=0$ and, because of $A$ is prime, $v=0$.
\end{proof}

\begin{lem}
Let $A$ be a prime superalgebra, $L$  a Lie ideal of $A$ such that $L$ is dense in A,   and $V$  a Lie subalgebra of $A$ such that $[V,L]\subseteq V$. If $v^2=0$ for every $v\in V_i$,  then $V_i=0$.
\end{lem}

\begin{proof}[Proof:]
Consider $l_0\in L_0$ and $a_0\in A_0$, then $[l_0,a_0]\in L$ and $[v,[l_0,a_0]]^2=0$ for every $v\in V_i$, that is 
$$(vl_0a_0-va_0l_0-l_0a_0v+a_0l_0v)^2v=0.$$
{\noindent  Expanding  yields}
$$vl_0a_0vl_0a_0v-vl_0a_0va_0l_0v-va_0l_0vl_0a_0v+va_0l_0va_0l_0v=0.$$
{\noindent Replacing $a_0$ by $a_0v$ gives $va_0vl_0va_0vl_0v=0$, and so,  $(vl_0va_0)^3=0$. Since $A_0$ is semiprime by Lemma 1.2,  it follows  from Lemma 1.1 in \cite {H1} that  $vl_0v=0$. Now let $l_1\in L_1$, we can prove in a similar way that $vl_1v=0$. Indeed, let $a_1 \in A_1$ and notice that $[v,[l_1,a_1]]^2=0$ and so}
$$(vl_1a_1+va_1l_1-l_1a_1v-a_1l_1v)^2v=0.$$
{\noindent Expanding and replacing $a_1$ by $a_1v$ give $va_1vl_1va_1vl_1v=0$, and so $(vl_1va_1)^3=0$. 
Therefore $vl_1vA_1$ is a right ideal of $A_0$. But $A_0$ is semiprime, by Lemma 1.2. So from Lemma 1.1 in \cite {H1} $vl_1vA_1=0$ and also $(vl_1v)(A_1+A_1^2)=0$. Since $A$ is prime we have $vl_1v=0$, and so $vLv=0$. Now, by Lemma 3.9, $V_i=0$.}
\end{proof}

Then, from now on, and until the end of this section, we will suppose that $A$ is a prime superalgebra, $I$ is a nonzero ideal of $A$ and $U$ is a  subalgebra of $A^-$ (that is, $U$ is a $\phi$-submodule of $A$ and $[U,U]\subseteq U$) such that   $[U, [I,I]]\subseteq U$. Moreover $U$ satisfies  the following conditions:  $u\circ v \in Z$ for every $u,v \in [U,U]_0$,  $u\circ v=0$ for
every $u,v\in [U,U]_1$ and $(u\circ v)^2=0$ for every $u\in [U,U]_0, v\in [U,U]_1$.  Now, we will consider the following set:
$$T=\{x\in A: [x,A]\subseteq U \}.$$

Since
 $$[[[U,[I,I]],[U,[I,I]]],A]\subseteq [[U,[I,I]],[[U,[I,I]],A]]\subseteq [[U,[I,I]],[I,I]]\subseteq U,$$

{\noindent we have $[[U,[I,I]],[U,[I,I]]]\subseteq T$. We notice that $T$ is a subring of $A$ because for every homogeneous elements $t,s\in
T$, from (2) } 
$$[ts,a]=[t,sa]+(-1)^{\bar t \bar s + \bar a \bar t}
[s, at]\in U.$$
{\noindent Let $T^\prime$ be the subring generated by $[[U,[I,I]],[U,[I,I]]]$. Since}
$$[[[U,[I,I]],[U,[I,I]]],[I,I]]  \subseteq [[U,[I,I]],[[U,[I,I]],[I,I]]] \subseteq [[U,[I,I]],[U,[I,I]]]$$ 
{\noindent  it follows that
$[T^\prime,[I,I]]\subseteq T^\prime$. We consider now two cases:

\begin{enumerate}

\item[{\rm (a)}]  $[T^\prime,[I,I]]=0$, 

\item[{\rm (b)}] $[T^\prime,[I,I]]\not=0$.

\end{enumerate}
 
We will suppose until the end of the section that  $A$ is neither commutative, nor a central order  in a 4-dimensional simple superalgebra, nor a central order in a 8-dimensional simple superalgebra. Hence, from Lemma 3.1,  we can suppose also  that  there exists  a nonzero ideal of $A$, $J$,  such that $J\subseteq \overline {[I, I]}$.
}

\bigskip

Lets go to consider the first case.

\bigskip

{\bf CASE (a): $[T^\prime, [I,I]]=0$}

\begin{lem}
In the above situation,  if $[T^{\prime}, [I,I]]=0$,  then $U\subseteq Z$.
\end{lem}

\begin{proof}[Proof:]

If  $[T^\prime, [I,I]]=0$, it follows  from (1) that $[T^\prime, J]=0$, and from Lemma 2.3 in \cite {L-S}, we deduce that $T^\prime \subseteq Z$. Therefore 
\begin{gather}
 [[U,[I,I]]_i,[U,[I,I]]_i]\subseteq Z,\\
    [[U,[I,I]]_0,[U,[I,I]]_1]=0
 \end{gather}
 
 We will prove that, with the above supposition of being $A$ neither commutative, nor a central order in a 4-dimensional simple superalgebra, nor a central order in a 8-dimensional simple superalgebra,  then $U\subseteq Z$.  We present the proof  in 7 steps.

\medskip

1. $[[U,U],[I,I]]_0\subseteq Z(A_0)$. From Lemma 1.2,  $A_0$ is semiprime, and $[I,I]_0$ is a Lie ideal of $A_0$,  $[[U,U],[I, I]]_0 \subseteq [I,I]_0$, $[[[U,U],[I,I]]_0,[I,I]_0]\subseteq [[U,U],[I,I]]_0$ and $[[[U,U],[I,I]_0,[[U,U], [I,I]]_0]\subseteq Z(A_0)$ because of (5). So we have the conditions of Lemma 1.1 in \cite {La-M} and therefore we can conclude by this Lemma  that $[[U,U],[I,I]]_0\subseteq Z(A_0)$.

\medskip

2. $[[U,U]_0,[I,I]_0]=0$. We notice that  $A_0$ is semiprime by Lemma 1.2,  $[I,I]_0$ is a Lie ideal of $A_0$ and $[[U,U]_0,[[U,U]_0,[I,I]_0]] =0$ by step 1. So we can  apply  Lemma 1.1  and  we obtain that  $[[U,U]_0,[I,I]_0]=0$.

\medskip

3. Either  $[[U,U], [I,I]]_0=0$ or $ [[U,U], [I,I]]_1=0$.  Let $u\in [[U,U], [I,I]]_0$ and $v\in [[U,U],[I,I]]_1$. By (6)  we have $uv=vu$. But, from the hypothesis, $(u\circ v)^2=0$, hence $4u^2v^2=0$. So, since $u\in [[U,U],[I,I]]_0\subseteq Z$ and $A$ is prime,  either $[[U,U],[I,I]]_0=0$ or  $v^2=0$ for every $v\in [[U,U], [I,I]]_1$. 
If $v^2=0$ for every $v\in [[U,U],[I,I]]_1$, from Lemma 3.3, taking $L=[I,I]$, $V=[[U,U],[I,I]]$, we deduce that $[[U,U],[I,I]]_1=0$.

\medskip

4. If $[[U,U], [I,I]]_0=0$ we claim that $[U,U] \subseteq Z$. We notice that
\begin{align*} 
[U,U]_1[I,I]_0& \subseteq   [[U,U]_1, [I,I]_0]+[I,I]_0[U,U]_1\subseteq [U,U]_1 + [I,I]_0 [U,U]_1, \\
[U,U]_1 [I,I]_1 &\subseteq  [[U,U]_1, [I,I]_1] + [I,I]_1 [U,U]_1\subseteq [I,I]_1 [U,U]_1,
\end{align*}
{\noindent because of the hypothesis of this step. Therefore  $[U,U]_1[I,I]\subseteq [U,U]_1 + [I,I] [U,U]_1$.  In general, we can prove by induction on $m$ that }
$$[U,U]_1 [I,I]^{m}\subseteq [U,U]_1 + \sum_i [I,I]^i [U,U]_1,$$
{\noindent  and so $[U,U]_1 J \subseteq [U,U]_1 + \sum_i [I,I]"^i[U,U]_1$.   Now since}
 $$[[U,U]_1, [[U,U], [I,I]]_1]\subseteq [[U,U]_1, [I,I]_1]\subseteq [[U,U], [I, I] ]_0=0.$$
 {\noindent because of the hypothesis of this step,  and $[U,U]_1  \circ [[U,U], [I,I]]_1=0$ because of our hypothesis, it follows that $[U,U]_1 [[U,U], [I,I]]_1=0$, and therefore}
  $$[U,U]_1 J [ [U,U], [I,I]]_1  =0.$$
  {\noindent But $A$ is prime, so either $[U,U]_1=0$ or $[[U,U], [I,I]]_1=0$.  If $[U,U]_1=0$, then $[[U,U]_0, [I,I]_1]=0$, and from the hypothesis of this step  $[[U, U]_0, [I,I]] =0$. But then, from (2),  $[[U,U]_0,J]=0$ and so, by Lemma 2.3 in \cite{L-S},  $[U,U]_0 \subseteq Z$ and  $[U,U]\subseteq Z$. If $[[U,U],[I,I]]_1=0$, since $[[U,U]_0, [I,I]_0]=[[U,U]_1, [I,I]_1]=0$ by the hypothesis of this step, we get   $[[U,U], [I,I]]=0$ and so $[[U,U], J]=0$, from (2), and as above $[U,U]\subseteq Z$.}  

\medskip

5. If $[[U,U],[I,I]]_0\not= 0$, then by step 3 $[[U,U],[I,I]]_1=0$ and also we claim that $[U,U]\subseteq Z$. We have $[[U,U]_0, [I,I]_1]=0$, and by step 2 $[[U,U]_0, [I,I]_0]=0$. Therefore, from (2), $[[U,U]_0, J]=0$, and, by Lemma 2.3 in \cite{L-S}, $[U,U]_0\subseteq Z$. From the hypotesis about $U$, for every $u\in [U,U]_0, v\in [U,U]_1$ we have $(u\circ v)^2=0$. So, since $(u\circ v)^2 = 4 u^2 v^2=0$, we obtain from the primeness of $A$ that either $[U,U]_0=0$ or $v^2=0$ for every $v\in [U,U]_1$. If $[U,U]_0=0$, then $[[U,U],[I,I]]_0\subseteq [U,U]_0=0$, a contradiction with our assumption. If $v^2=0$ for every $v\in [U,U]_1$, from Lemma 3.3 applied to $[I,I]$ and $[U,U]$ we obtain that $[U,U]_1=0$ and so $[U,U]\subseteq Z$.

\medskip

6. If $I\cap Z\not=0$, then $U\subseteq Z$. Indeed, if $I\cap Z\not=0$, we can consider $Z^{-1}A$, which is a prime superalgebra over the field $Z^{-1}Z$. Since $Z^{-1} I \cap Z^{-1} Z \not=0$ it holds that $Z^{-1}I= Z^{-1}A$ and so $Z^{-1} ZU$ is a Lie ideal of $[Z^{-1} A, Z^{-1}A]$. From Theorem 3.3 in \cite {M}, and since  $A$ is a central order  neither in a commutative algebra, nor a 4-dimensional simple superalgebra, nor a 8-dimensional superalgebra, we deduce that either  $Z^{-1} ZU \subseteq Z^{-1} Z$ or there exist a nonzero ideal of $Z^{-1} A$, $Z^{-1}N$, such that $[Z^{-1}N, Z^{-1}A] \subseteq Z^{-1} ZU$. In the last case, since, by steps 4 and 5, $[U,U]\subseteq Z$, we have that $[[Z^{-1}N, Z^{-1}A],[Z^{-1}N, Z^{-1} A]]\subseteq  Z^{-1}Z$. We notice that if $[[Z^{-1}N, Z^{-1}A],[Z^{-1}N, Z^{-1}A]]=0$, from Lemma 1.5 and its proof and our hypothesis, $[Z^{-1}N,Z^{-1}A]\subseteq Z^{-1}Z$, and then $Z^{-1}N = Z^{-1} A$ if $[Z^{-1}N, Z^{-1}A]\not= 0$. And if $[Z^{-1} N, Z^{-1}A]=0$, by Lemma 2.3 in \cite {L-S}, $Z^{-1}N\subseteq Z^{-1}Z$ and $(Z^{-1}N)_1=0$, a contradiction with the primeness of $Z^{-1}A$. If $[[Z^{-1}N, Z^{-1}A],[Z^{-1}N, Z^{-1}A]]\not=0$, then also $Z^{-1}N=Z^{-1}A$. So $[[Z^{-1}A, Z^{-1}A],[Z^{-1}A, Z^{-1}A]]\subseteq Z^{-1}Z$ and then the superalgebra $Z^{-1}A$ verifies the identity $[Z^{-1}A, [[Z^{-1}A, Z^{-1}A],[Z^{-1}A, Z^{-1}A]]]=0$. Now by Lemma 2.6 in \cite{M} we have a contradiction with our supposition of  $A$ not being a central order neither in a commutative algebra, nor a 4-dimensional simple superalgebra, nor a 8-dimensional superalgebra (notice that the product $\circ$ in \cite {M} is our product $[\ , \ ]$ in the odd part). So $Z^{-1} ZU\subseteq Z^{-1} Z$ and then $U\subseteq Z$.

\medskip

7. If $I\cap Z= 0$, then $U\subseteq Z$. We consider $[U, [I,I]]$ and we notice that $[[U, [I,I]], [U,[I,I]]]\subseteq [U,U]\cap I \subseteq Z \cap I=0$. Therefore for every $ v\in [U, [I,I]]_1$, $v^2 =[v,v]\in [[U,[I,I]], [U,[I,I]]]=0$. From Lemma 3.1 applied to $[U,[I,I]]$ and $[I,I]$ it follows that $[U,[I,I]]_1=0$. Now let $u\in U_0$, then $[u, [u, [I,I]_0]] \subseteq [U,U] \cap I \subseteq Z\cap I=0$. By Lemmata 1.1 and 1.2 it is deduced that $[u,[I,I]_0]=0$, that is, $[U_0, [I,I]_0]=0$. Now we have $[U_0, [I,I]]=0$, and therefore, from (2), $[U_0, J]=0$. So $U_0\subseteq Z$ because of Lemma 2.3 in \cite{L-S}. But we have proved that $[U_1, [I,I]_0]\subseteq [U, [I,I]]_1=0$, and now we have $[U_1,[I,I]_1]\subseteq U_0 \cap I \subseteq Z \cap I =0$, therefore $[U_1, [I,I]]=0$, and then, from (2),  $[U_1, J]=0$.  Again,  by Lemma 2.3 in \cite {L-S}, $U_1\subseteq Z$ and $U\subseteq Z$.

\end{proof}

\bigskip

Now we consider the second case. 

\bigskip

{\bf CASE b): $[T^\prime, [I,I]] \not= 0$.}

\bigskip

We recall that $[T^\prime, [I,I]]\subseteq T^\prime$.
\begin{lem}
If $[T^{\prime},[I,I]]\not=0$, then either  $T^\prime$ is dense or  $[t,u],[u,s]=0$ for every $u\in [T^{\prime}, [I,I]]_i,$ such that $[u, u]=0$, and $ t, s \in T^{\prime}$, homogeneous .
\end{lem}

\begin{proof}[Proof:]
We notice that if $u\in [T^{\prime}, [I,I]]_0,$  then $[u, u]=0$, and if $u\in [T^{\prime}, [I,I]]_1,$ then $[u, u]=0$ is equivalent to $u^2=0$.  We will prove that if  there exists $t, s \in T^{\prime}$, homogeneous,  and $u\in [T^{\prime}, [I,I]]_i,$ with $[u, u]=0$,  such that  $[t,u],[u,s]\not=0$, then $T^{\prime}$ is dense. First we see that  $[t,u][u,s]A\subseteq T^\prime$. We have $[u,s]a= [u,sa]-(-1)^{\bar s \bar u}s[u,a]$, from (2),  for every homogeneous element $a\in A$, therefore $[t,u][u,s]a=[t,u][u,sa]-(-1)^{\bar u \bar s}[t,u]s[u,a]$. But
$$[t,u][u,sa]=[t,u[u,sa]]-(-1)^{ \bar t \bar u}u[t,[u,sa]]=(-1)^{\bar u}[t,[u,usa]]-(-1)^{\bar t \bar u}u[t,[u,sa]]\in T^\prime,$$
{\noindent because $T^\prime $ is a subring and $[I,I]$ is a Lie ideal of $A$. And also }
$$[t,u]s[u,a]=[t,u][s,[u,a]]+(-1)^{\bar s (\bar u+ \bar a)}[t,u][u,a]s \in T^\prime , $$
{\noindent because }
$$ [t,u][u,a]=[t,u[u,a]]-u[t,[u,a]]= [t,[u,ua]]-u[t, [u,a]] \in T^\prime.$$
{\noindent Therefore $[t,u]s[u,a]\in T^\prime$, and also $[t,u][u,s]a \in T^\prime$ for every $a\in A, u\in [T^\prime, [I,I]]_i, t,s$ homogeneous elements in $ T^\prime$.}

Next we will show that $\overline{[I,I]} [t,u][u,s]A\subseteq T^\prime$. Since $[T^\prime, [I,I]] \subseteq T^\prime$ it follows that $[I,I][t,u][u,s]A\subseteq [[I,I],[t,u][u,s]A]+[t,u][u,s]A[I,I]\subseteq T^\prime$. Notice that also $[I,I]^2[t,u][u,s]A\subseteq [[I,I],T^\prime]+[t,u],[u,s]A \subseteq T^\prime$. Using induction over $i$ it is easy to prove that $[I,I]^i[t,u][u,s]A\subseteq T^\prime$, and so that $J[t,u][u,s]A\subseteq T^\prime$. Therefore either $T^\prime$ is dense in $A$ or, because of the primeness of $A$, $[t,u][u,s]=0$ for every $t,s \in T^\prime, u\in [T^\prime, [I,I]]_i$.

\end{proof}

\begin{lem}
If $T^\prime$ is dense, then $A$ is a central order in a superalgebra $B$ satisfying the condition $[[B,B],[B,B]]_1 \circ [[B,B].[B,B]]_1=0$.
\end{lem}

\begin{proof}[Proof:]
Let $N=J[t,u][u,s]A \not=0$, since $[T^\prime, A]\subseteq U$ and $N\subseteq T^\prime$, we have $[N,A]\subseteq U$. From the hypothesis about $U$ $u\circ v\in Z$ for every $u, v \in [U,U]_0$, therefore $u\circ v \in Z$ for every $u, v \in [[N,A], [N,A]]_0$.

 We suppose first  that $u\circ v=0$ for every $u, v \in [[N,A],[N,A]]_0$. Then $1/2 (u\circ u)= u^2 =0$ for every  $u \in [[N,A], [N,A]]_0$, and since $A_0$ is semiprime because of Lemma 1,2, then  it follows from Lemma 1 in \cite {La-M} that $[[N,A],[N,A]]_0=0$. Therefore $L=[[N,A],[N,A]]$ is a Lie ideal of $A$ such that $[L,L]=0$ and from Theorem 3.2 in \cite {M} we get that $[[N,A],[N,A]]\subseteq Z$  (because $A$ in neither commutative, nor a central order in a 4-dimensional simple superalgebra, nor a central order in a 8-dimensional simple superalgebra). But then $[[N,A],[N,A]]=0$, and again from Theorem 3.2 in \cite{M},  $[N,A]$ is a Lie ideal of $A$ such that $[N,A]\subseteq Z$, and so $[N,A]_1=0$. Hence $[N_1, A_1^2]=0$ and since $A_1^2\not=0$ because $A$ is prime we obtain that $N_1\subseteq Z(A)$ because of Lemma 1.3. Besides also $[N_0,A_1]=0$ and so $[N_0, A_1^2]=0$, from (2),  what means that $[N, A_1+A_1^2]=0$ and from Lemma 2.3 in \cite {L-S} it follows that $N_0\subseteq Z$. So we have  a nonzero ideal, $N$,  of $A$ such that $N\subseteq Z(A)$, with $A$ prime, and we can deduce that $A$ is commutative like in the proof of Lemma 3.1, that is a contradiction with our hypothesis.
 
Therefore there exists $u, v\in [[N,A],[N,A]]_0$ such that $0\not= u\circ v \in Z$. Then we may form the localization $Z^{-1}A$. Since $[N,A]\subseteq U$ we have 
$$[[Z^{-1}N, Z^{-1}A], [Z^{-1}N, Z^{-1}A]] \subseteq [Z^{-1}ZU, Z^{-1}ZU],$$
{\noindent and so from the hypothesis about $U$, for every $u, v \in  [[Z^{-1}N, Z^{-1}A], [Z^{-1}N,  $ $Z^{-1}A]]_0$ we get $u\circ v \in Z^{Ð1}Z \cap Z^{-1}N$.  But $Z^{-1}Z$ is a field and so $Z^{-1}N$ has some invertible element forcing $Z^{-1}N=Z^{-1}A$. Therefore $[Z^{-1}N, Z^{-1}A] = [Z^{-1}A, Z^{-1}A]\subseteq Z^{-1}(ZU)$ and again, by the hypothesis about $U$,  it follows that $[[Z^{-1} A, Z^{-1} A],$ $[Z^{-1} A, Z^{-1} A]]_1 \circ [[Z^{-1} A, Z^{-1} A],[Z^{-1} A, Z^{-1} A]]_1=0$. }

\end{proof}

We now will study superalgebras $B$ satisfying the condition $[[B,B], [B, B]]_1\circ [[B,B], [B, B]]_1 =0$. We notice that the ${\bf Z}_2$-grading is given by the automophism $\sigma$ of the algebra defined by $x_i^\sigma = (-1)^i x_i$, on homogeneous elements $x_i$. Then, we have the group of automorphisms $G= \{ 1, \sigma\}$ acting on $A$. Superidentities in $B$ are then  special types of $G$-identities, as defined in \cite {K}, that is identities involving elements and images of elements under the action of $G$, a group of automorphisms. Therefore we can apply results about $G$-identities, in particular the following one due to V.K. Kharchenko. We denote by $R^G$ the set of elements fixed under every automorphism of $G$.

\begin {pro}(\cite {K}, Theorem 1)
Suppose $G$ is a finite commutative group of automorphisms of a semiprime algebra $R$ over a commutative domain $K$ containing a primitive root of degree $n=|G|$, and suppose that $R^G$ is prime. If $R$ satisfies a nontrivial $G$-identity, then $R$ is a PI-algebra.
\end{pro}

In our case, the group $G$ has two elements, and here the conditions on $K$ always hold because we can consider our algebras as ${\bf Z}$-algebras, and every automorphism is ${\bf Z}$-automorphism. Then,  an algebra satisfying $[[B,B], [B, B]]_1\circ [[B,B], [B, B]]_1 =0$ is in fact a PI-algebra,  if $B^G=B_0$ is a prime algebra. We consider next the case when $B_0$ is not prime.

\begin{lem}
If $B$ is a prime superalgebra and satisfies  the condition $[[B,B], [B, B]]_1\circ [[B,B], [B, B]]_1 =0$, and $B_0$ is not prime, then $B_0$ is commutative and $B$ is a PI-algebra.
\end{lem}

\begin{proof}[Proof:]
 By Lemma 1.5 in \cite {M}, $B$ has an ideal $I$ which is a Morita superalgebra. This means that we have a Morita context $(R, S, M, N, \mu, \tau)$, where $R$ and $S$ are associative $\phi$-algebras, $M$ is an $R$-$S$-bimodule, $N$ is an $S$-$R$-bimodule and $\mu \colon M\otimes _RN \to R, \tau \colon N\otimes _R M \to S$ are bimodule homomorphisms, such that $I$ is the set of matrices
 $$I= \begin{pmatrix} R & M \\ N & S \end{pmatrix},$$
with the known algebra structure given by the Morita contex, and the following grading as superalgebra 
$$I_0=\begin{pmatrix} R & 0 \\ 0 & S\end{pmatrix}, I_1= \begin{pmatrix} 0 & M \\ N & 0 \end{pmatrix}.$$
Moreover,  Lemma 1.5 in \cite {M} and its proof says that $R$ and $S$ are prime algebras and orthogonal ideals of $B_0$, and also that $I_0$ intesects nontrivially every nonzero ideal of $B_0$.

We have that $I$ satisfies $[[I,I], [I, I]]_1\circ [[I,I], [I, I]]_1 =0$, hence
\begin{equation*} \begin{split}
&[[\begin{pmatrix} R & 0 \\ 0 & 0 \end{pmatrix},  \begin{pmatrix} R & 0 \\ 0 & 0\end{pmatrix}], [\begin{pmatrix} R & 0 \\ 0 & 0 \end{pmatrix},  \begin{pmatrix} 0 & M \\ 0 & 0\end{pmatrix}]]\\ 
& \circ [[\begin{pmatrix} R & 0 \\ 0 & 0 \end{pmatrix},  \begin{pmatrix} R & 0 \\ 0 & 0\end{pmatrix}], [\begin{pmatrix} R & 0 \\ 0 & 0 \end{pmatrix},  \begin{pmatrix} 0 & 0 \\ N & 0\end{pmatrix}]]=0.
\end{split} \end{equation*}

So $[R,R]RMNR[R,R]=0$, but since $R$ is prime either $[R,R]=0$ or $RMNR=0$. If $RMNR=0$, then $MN=0$ because $R$ is prime, and  so $NM$ is a trivial ideal of $S$, which is also prime, therefore $NM=0$ and $I_1$ is a trivial ideal of $I$. But then $II_1I$ is also a trivial ideal of $B$, and because $B$ is prime and $I\not=0$ we have $I_1=0$, and as a consequence $R$ and $S$ are orthogonal ideals of $A$, a contradiction with the primeness of $B$. Thus $[R,R]=0$. Similarly we can prove that $[S,S]=0$, and so $I_0$ is commutative. But then for every $y, z \in I_0$ and $a,b\in B_0$ it follows, from (2) and (1),  that
$$ y[a,b]z=y[a,bz]-yb[a,z]= [ya,bz]-[y,bz]a-[yba,z]+[yb,z]a=0,$$
and so $ [([B_0, B_0],B_0 + [B_0, B_0]) I_0]^2\subseteq I_0 [B_0, B_0]I_0=0$.  But $B_0$ is semiprime because of Lemma 1.2, and $[B_0,B_0]I_0$ is an ideal of $B_0$, therefore $( [B_0, B_0]B_0 + [B_0, B_0])I_0=0$. Since $I_0$ intersects nontrivially every nonzero ideal of $B_0$ we have $([B_0, B_0]B_0 + [B_0, B_0]) \cap I_0 = N\not= 0$ satisfies that $N^2=0$. So $[B_0, B_0]B_0 + [B_0, B_0]=0$, and $B_0$ is commutative.

\end{proof}

\begin{lem}
If $B$ is a prime superalgebra satisfying  the condition $[[B,B],[B,B]]_1\circ [[B,B],[B,B]]_1$ $=0$, then $B$ is a central order in $\Omega \oplus \Omega . v$ with $v^2\in \Omega$ (where $\Omega$ is the field of fractions of $Z$):
\end {lem}

\begin{proof}[Proof:]
By Proposition 3.5 and Lemma 3.6,  $B$ is PI-algebra. Then by Lemma 1.7 in \cite {M} $B$ is a central order in a simple superalgebra which is finite dimensional over $\Omega $, that is, $C= Z^{-1} B$ is simple and finite dimensional over $\Omega = Z^{-1} Z$. Take $\bar \Omega$ an algebraic closure of $\Omega$. Then $\bar C= Z^{-1}B\otimes \bar \Omega$ is a simple superalgebra, finite dimensional over $\bar \Omega$ and satisfies $[[\bar C, \bar C], [\bar C, \bar C]]_1\circ [[\bar C, \bar C],[\bar C, \bar C]]_1=0$.  But finite dimensional simple associative superalgebras were classified in \cite {W} and over an algebraically closed field we obtain that $\bar C= \bar \Omega \oplus \bar \Omega . u$ with $u^2=1$, so $C = \Omega \oplus \Omega . v$ with $v^2\in \Omega$.

\end{proof}

So, from Lemmata 3.6, 3.7, 3.8 and 3.9 we can deduce:

\begin{coro}
If $T^{\prime} $ is dense in $A$,  then $A$ is either commutative, or a central order  in a 4-dimensional simple superalgebra, or a central order in a 8-dimensional simple superalgebra
\end{coro}

So, because of our assumption,  $T^\prime$ is not dense in $A$, and then $[t,u][u,s]=0$ for every $u\in [T^{\prime}, [I,I]]_i$ such that $[u,u]=0$, and for every $ t,s\in T^{\prime}$, homogeneous.  We will prove that this can not occur under our assumption of  $[T^\prime, [I,I]]\not=0$ and of $A$ being neither commutative, nor a central order in a 4-dimensional simple superalgebra, nor a central order in a 8-dimensional simple superalgebra

\begin{lem}
If $[t,u][u,s]=0$ for every $u\in [T^\prime, [I,I]]_i$ such that $[u, u]=0$, and for every $ t,s \in T^\prime$, then $[T^{\prime}, [I,I]]=0$.
\end{lem}

\begin{proof}[Proof:]
We prove the result in 4 steps. 

\medskip

1. $[X, X]= [X_1, X_1]$  with $X=[T^\prime, [I,I]]$. Indeed, we have $[t,u]^2=0$ for every $t\in T^\prime, u\in X_0$. Let $x, y \in X$, homogeneous, and $u\in X_i$ such that  $u^2=0$. From our assumption $[u,x][u, y]=0$, and expanding this gives $ux u y- (-1)^{\bar y \bar u}u x yu+(-1)^{\bar x \bar u + \bar y \bar u}xuyu=0$. Right multiplication by $u$ gives $uxu yu=0$. Since $[y, l]\in X$ for every $l\in [I,I]_i$, we obtain that $[y, [l,u]]\in X$. So $ux u [y, [l,u]]u=0$. Expanding this expression yields $ux uluyu=0$.  From Lemma 3.2 we deduce that $uXu=0$. If $u, u^\prime \in X$ are homogeneous elements  with $u^2=(u^\prime)^2=0$, we conclude that
$$(uu^\prime)^2 = uu^\prime uu^\prime \in uXuu^\prime=0.$$
{\noindent If $l\in [I,I]_i$ we have }
$$ 0=u[u^\prime, l]uu^\prime= uu^\prime l u u^\prime , $$
{\noindent so $uu^\prime[I,I]uu^\prime=0$ and from Lemma 3.2, }
$$uu^\prime =0 \quad \text {for every}  \quad u, u^\prime \in X, \text { homogeneous,  with} \quad u^2=(u^\prime )^2 =0. \quad (*)$$
{\noindent  Now consider $x, y \in X_1, u, v\in X_0$. We have $[x,u]^2=0=[y,v]^2$, and so $[x,u][y,v]=0$, because of $(*)$. Since $[X_0, X_1]$ is additively generated by the elements $[x,u]$ with $x\in X_1, u \in X_0$, we have $v^2=0$ for every $v\in[X_0, X_1]$. From Lemma 3.3, $[X_0, X_1]=0$, and $[X,X]=[X_0,X_0] + [X_1, X_1]$. Now consider $K=[X_0, X_0]$. We notice that $K$ is a subalgebra of $A^-$ and $[K, [I,I]]\subseteq K$. From our assumption for every $ x, y, u, v \in X_0$ we have $[x,u]^2=[y,v]^2=0$, and so, by $(*)$ we obtain that $[x,u] [y,v]=0$. Again, since $[X_0,X_0]$ is additively generated by the elements $[x,u]$ with $x, u \in X_0$, we  deduce that for every $v\in K=K_0$, $v^2=0$. From Lemma 3.3, $K=0$. Therefore $[X,X]= [X_1, X_1]$.
}

  \medskip
  
  2. $ [ W_0, [I,I]_0]=0,$ and $  B_1=0$ with $B=[W, [I,I]], W=[S,S]$ and $S=[U, [I,I]]$. We have $ B\subseteq [U,U]$. Since $W\subseteq T^\prime, B=[W,[I,I]]\subseteq [T^\prime, [I,I]]= X$.  So, from step 1,  for every $b_0, b_0^{\prime} \in B_0, [b_0, b_0 ^{\prime}]=0$.  But $B\subseteq [U,U]$, and then $b_0\circ b_0^\prime \in Z $  for every $b_0\in B_0$. Therefore $B_0^2\subseteq Z$.  Now Lemma 4 in \cite {H3} yields $[B_0,[I,I]_0]=0$. Hence $[[W_0,[I,I]_0],[I,I]_0]=0$, and by theorem 1 in \cite {H3} $[W_0, [I,I]_0]=0$. Also, since $B\subseteq [U,U]$, we have $ (b_0\circ b_1)^2=0$ for every $b_0\in B_0, b_1\in B_1$. But $B\subseteq X$, and so $ [b_0, b_1]=0$. Now applying that $B_0^2\subseteq Z$, we obtain that $b_0^2b_1^2=0$ for every $b_0\in B_0, b_1 \in B_1$. We consider now the ideal of $A$, $b_0^2A$ and then $b_1^2(b_0^2A)=0$. Hence, from the primeness of $A$, either $b_0^2=0$ or $b_1^2=0$. From Lemma 3.3,  if $b_0^2=0$,  $B_0=0$ and then $[B_1,B_1]=0$ and therefore $[b_1, b_1]=b_1^2=0$.  So in any case $b_1^2=0$ for every $b_1\in B_1$, and again from Lemma 3.3 $B_1=0$.
  
  \medskip
  
  3.   $W_1=0$ with $W=[S,S], S=[[U,[I,I]]$. Since  $W\subseteq [U,U]$,  $w_0 \circ w_0^\prime \in Z$  for every $w_0, w_0^\prime \in W_0$.  But $W_0\subseteq [I,I]_0$, and then by step 2 $[w_0, w_0^\prime]=0$. Therefore  $W_0^2\subseteq Z$. Moreover,  since $W\subseteq [U,U]$, we have also $(w_0\circ w_1)^2=0$, and because of step 2 also $ [w_0,w_1]\in B_1=0$, for every $w_0\in W_0, w_1 \in W_1$. So  $w_0^2w_1^2=0$. Hence $w_0^2 A$ is an ideal of $A$ such that $w_0^2A w_1^2=0$, and then either $w_0^2=0$ or $w_1^2=0$. Now, as in the proof of step 2, we can deduce that $W_1=0$.
  
  \medskip
  
  4. $B=0$ and then $[T^\prime, [I,I]]=0$, a contradiction. Indeed, from step 2, $B_1=0$. And since $B=[W, [I,I]]$, from step 2 and 3, $B= B_0=[W_0, [I,I]_0] + [W_1, [I,I]_1]=0$. But $T^\prime$ is the subring of $T$ generated by $W=[S,S]=[[U,[I,I]], [U,[I,I]]]$. So if $[w,y]=0$ for every $w\in W, y\in [I,I]$, homogeneous, then,  since $[ w^{\prime}w, y]= w^{\prime}[w, y] + (-1)^{\bar w \bar y} [w^{\prime}, y] w=0$, we deduce that $[T^\prime, [I,I]]=0$.
\end{proof}

So, in the last Lemma we have arrived to a contradiction with our assumption in case b): $[T^\prime, [I,I]]=0$. Hence, from the above results, we can  deduce the following theorem

\begin{theo}
Let $A$ be a prime superalgbra and let $I$ be a nonzero proper ideal of $A$. Suppose that $U$ is a  subalgebra of $A^-$ such that $[U, [I,I]]\subseteq U$, $u\circ v \in Z$ for every $u, v \in [U,U]_0$, $u\circ v=0$ for every $u, v\in [U,U]_1$ and $(u\circ v)^2=0$ for every $u\in [U,U]_0, v\in [U,U]_1$. Then either $A$ is commutative, or $A$ is a central order in a 4-dimensional simple superalgebra, or $A$ is a central order in a 8-dimensional simple superalgebra, or $U\subseteq Z$.
\end{theo}

So the prime images of Lie ideals $U$ of $[K,K]$ satisfying  that $[u\circ v, w]=0$ for every $u,v \in [U,U], w\in U$ when the prime ideal $P$ satisfies that $P^*\not= P$ are like this.

\begin{coro}
Let $A$ be semiprime, and let $U$ be a Lie ideal of $[K,K]$ such that $[u\circ v, w]=0$ for every $u,v \in [U,U], w\in U$. If $P$ is a prime ideal of $A$ such that  $P^*\not= P$ then either the projection of $U$ in $A/P$ is central, or $A$ is commutative, or $A$ is a central order in a 4-dimensional simple superalgebra or in a 8-dimensional simple superalgebra.
\end{coro}

\bigskip

\section {Prime images of Lie ideals when $P^*= P$.}

\bigskip

Next we consider the cases when $P^*= P$, for $P$ a prime ideal of $A$. So wehave a superinvolution on $A/P$ induced by the superinvolution on $A$. Recall that a superinvolution on $A$ is said to be of the first kind if $Z_H=Z$, and it is said to be of the second kind if $Z_H\not=Z$.

\begin{lem}
Let $A$ be a prime superalgebra with a superinvolution $*$ of the second kind. Let $U$ be a Lie ideal
of $[K,K]$ such that $u\circ v \in Z$ for every $u,v \in [U,U]_0$,  $u\circ v=0$ for every $u,v\in [U,U]_1$ and $(u\circ v)^2=0$ for every $u\in [U,U]_0, v\in [U,U]_1$.
Then either $U\subseteq Z$ or $A$ satisfies $S(3)$. 
\end{lem}

\begin{proof}[Proof:]
If $*$ is of the second kind we know that $Z_H=\{x\in Z : x^*=x\}\not= Z$. We may localize $A$ by
$V= Z_H - \{0\}$ and replace $U$ by $V^{-1}(Z_HU)$ and $A$ by $V^{-1}A$. The hypothesis remains unchanged, so we
keep for this superalgebra the same notation $A$, and now
$Z$ is a field. Let
$0\not= t\in Z_K$. Then $H=tK$ and $A=tK+K$. It follows that
$[ZU, [A,A]]\subseteq ZU$,  $u\circ v \in Z$ for every $u,v \in Z[U,U]_0$, $u\circ v=0$ for
every $u,v\in Z[U,U]_1$ and $(u\circ v)^2=0$ for every $u\in [U,U]_0, v\in [U,U]_1$. By theorem 3.11, either $ZU\subseteq Z$, which implies that $U\subseteq Z$, or $A$
satisfies $S(3)$.
\end{proof}

\begin{lem}
Let $A$ be a prime superalgebra with a superinvolution $*$ of the first kind. Let $U$ be a Lie ideal
of $[K,K]$ such that $u\circ v \in Z$ for every $u,v \in [U,U]_0$,  $u\circ v=0$ for every $u,v\in [U,U]_1$ and $(u\circ v)^2=0$ for every $u\in [U,U]_0, v\in [U,U]_1$.
Then either $U=0$ or $A$ satisfies $S(4)$. 
\end{lem}

\begin{proof}[Proof:]
Since  $*$ is of the first kind, $Z_K= K \cap Z=0$. So, from Theorem 4.1 and Lemma 4.1 in \cite {Go-L-S}, either $[K,K]$ is dense in $A$ or $A$ satisfies $S(2)$.  If $u^2=0$ for every $u\in [U,U]_0$, applying Theorem 3.3 in \cite {Go-L-S} we obtain that $[U,U]=0$. But then by Lemma 4.5 in \cite {Go-L-S} we obtain that either $U=0$ or $A$ satisfies $S(2)$. Suppose
then that $u^2\not= 0$ for some $u\in [U,U]_0$. By Theorem 3.4 in \cite {Go-L-S} we get that either
$[U,U]\subseteq Z$ or $A$ satisfies $S(4)$. But if $[U,U]\subseteq Z$ then $[[U,U],[U,U]]=0$ and applying twice Lemma 4.5 in \cite {Go-L-S} yields $U=0$.

\end{proof}

\bigskip

Combining the above results we obtain

\bigskip

\begin{theo}
Let $A$ be a semiprime superalgebra and $U$ a Lie ideal of $[K,K]$ with $u\circ v \in Z$ for every $u,v
\in [U,U]_0$,  $u\circ v=0$ for every $u,v\in [U,U]_1$ and $(u\circ v)^2=0$ for every $u\in [U,U]_0, v\in [U,U]_1$. Then $A$ is the subdirect sum of two semiprime
homomorphic images $A^{\prime}$, $A^{\prime \prime}$, such that $A^\prime$ satisfies $S(4)$ and the
image of $U$ in $A^{\prime \prime}$ is central.
\end{theo}

\begin{proof}[Proof:]
Let $T^\prime=\{ P: P$ is a prime ideal of $A$ such that $A/P$ satisfies $S(4)\}$ and let
$T^{\prime \prime}= \{ P : P$ is a prime ideal of $A$ such that the image of $U$ in $A/P$ is
central$\}$. 

If we consider $P$ a prime ideal of $A$ such that $P^*\not= P$, we know from Corolally 3.13
that either $A/P$  is a central order in a simple superalgebra at most 8-dimensional
over its center, or $(U+P)/P$ is central. If we consider $P$ a prime ideal of
$A$ such that $P^*=P$, it follows from Lemmata 4.1, 4.2  that either $A/P$  is a central
order in a simple superalgebra at most 16-dimensional over its center, or the image of $U$ in $A/P$ is
central. 

So every prime ideal of $A$ belongs either $T^\prime$ or $T^{\prime \prime}$. Then
$A^\prime$ is obtained by taking the quotient of $A$ by the intersection of all the prime ideals in
$T^\prime$, and $A^{\prime \prime}$ is obtained by taking the quotient of $A$ by the intersection of
all the prime ideals in $T^{\prime \prime}$. This proves the theorem.
\end{proof}

\bigskip

We finally arrive at the main theorem on the Lie structure of $[K,K]$.

\bigskip

\begin{theo}
Let $A$ be a semiprime superalgebra with superinvolution $*$, and let $U$ be a Lie ideal of $[K,K]$. Then
either $A$ is a subdirect sum of two semiprime homomorphic images $A^\prime$, $A^{\prime \prime}$,
with $A^\prime$ satisfying $S(4)$ and the image of $U$ in $A^{\prime \prime}$ being central, or $U\supseteq
[J\cap K, K]\not= 0$ for some ideal $J$ of $A$.
\end{theo}

\begin{proof}[Proof:]
We consider $V=[U, U]$, which is also a Lie ideal of $[K,K]$. From Lemmata 2.1 and 2.3 we know that either $V$ is dense in $A$, and so there exist a
nonzero ideal $J$ such that $J\subseteq \bar V$, or the conditions i), ii) and iii) in Lemma 2.3 are satisfied by $V$. In the second case we obtain by Theorem 4.3 the first part of
the theorem for $V$. So $(V+P)/P$ is central in $A/P$ for some $P$ prime ideals of $A$. But we notice that in this if $(V+P)/P\subseteq Z(A/P)$, then the conditions i), ii) and iii) in Lemma 2.3 are also satisfied by $(U+P)/P$ in $A/P$. So from Corollary 3.13 and Lemmata 4.1 and 4.2 we have that either $(U+P)/P$ is central or $A/P$ verifies $S(4)$. So,  like in Theorem 4.3, we have the first part of the theorem.  Now we assume  that $J\subseteq \bar V$.

The identity
$$[xy,z]=[x,yz]+(-1)^{\bar x \bar y+\bar x \bar z}[y,zx]$$

{\noindent can be used to show that $[\bar V,A]=[V,A]$. Hence $[J\cap K, K]\subseteq [\bar V,
A]=[V,A]=[V,H]+[V,K]$. But $[V,H]\subseteq H$, and $[V,K]\subseteq K$, so $[J\cap K, K]\subseteq [[U,U],K]
\subseteq U$.}

Finally, suppose that $[J\cap K, K]=0$, then $[u\circ v,w]=0$ for every $u,v \in [J\cap K, J\cap K], w\in J\cap K$ because
$[uv,w]=u[v,w]+(-1)^{\bar v
\bar w}[u,w]v=0$. So by Lemmata 4.1, 4.2 and Corollary 3.13 it follows
that for each prime image, $A/P$, of $A$ either  its center contains
$((J\cap K)+P)/P$, or $A/P$ is a central order in a simple superalgebra at most 16-dimensional over
its center.

We claim that if the image of $J \cap K$ in $A/P$ for some prime ideal $P$ of $A$ is central, then $A$ is as described in the first part of the conclusion of the theorem. 

Let $P$ be a prime ideal such that $P^*\not= P$. If $(J+P)/P\not=0$,
then since $A/P$ is a prime superalgebra we get $((J\cap P^*) +P)/P \not=0$, and so we have $((J\cap
P^*) +P)/P\subseteq ((J\cap K)+P)/P \subseteq Z_0(A/P)$, that is, $A/P$ is commutative. So $A/P$ is
commutative unless $J\subseteq P$. And if $J\subseteq P$, then by the proof of Lemma 2.1  we know
that  $A[u\circ v, w]A[u\circ v,w]A[u\circ v, w]A\subseteq P$ for every $u, v \in [V,V],  w \in V$. Because  $P$
is a prime ideal we deduce that $[u\circ v,w]\in P$ for every $u, v \in [V,V], w \in V$. But now, by Lemma 2.3,  $(V+P)/P$ satisfies the conditions i), ii) and iii) in this lemma.   So by Corollary 3.13 we obtain that either
$(V+P)/P\subseteq Z_0(A/P)$, or $A/P$ satisfies $S(4)$. Applying the same to $(U+P)/P$, which satisfies also that  $[u\circ v, w]=0$ for every $u, v\in ([U,U]+P)/P, w\in (U+P)/P$, we obtain that either $(U+P)/P\subseteq Z_0(A/P)$ or $A/P$ satisfies $S(4)$.

 And, if $P$ is a prime ideal such that $P^*=P$, then $A/P$ has a superinvolution
induced by * and
$K(A/P)=(K+P)/P$.  In this case, if $((J\cap K) +P)/P=0$, we get $(J+P)/P\subseteq (H+P)/P=H(A/P)$, and
therefore $(J+P)/P$ is supercommutative. But then for every $a,b \in A/P$ and $y,z \in (J+P)/P$ it
follows that

\begin{eqnarray*}
yabz&=&(-1)^{(\bar b + \bar z)(\bar y + \bar a)}(bz)(ya)=(-1)^{\bar b (\bar y +
\bar a)} b(ya)z\\
&=& (-1)^{\bar b \bar y + \bar b \bar a+ (\bar a +\bar z)\bar
y} b(az)y= (-1)^{\bar b \bar a} ybaz.
\end{eqnarray*}

{\noindent Since $A/P$ is prime, $ab = (-1)^{\bar a \bar b} ba$, that is, $A/P$ is
supercommutative. Now, from Lemma 1.9 in \cite {M},  $A/P$ is a central order in a simple superalgebra
at most $4$-dimensional over its center. And, if
$((J\cap K)+P)/P\not=0$, then
$Z_0(A/P)\not=0$. So by localizing at
$V=(Z_0(A/P)\cap H(A/P))-\{0\}$  we can suppose that $Z_0(A/P)$ is a field, which we denote by
$Z$. We will replace $V^{-1}(A/P)$ by $A/P$ and $V^{-1}((J+P)/P)$ by $(J+P)/P$. Then, if
$0\not= t\in ((J\cap K)+P)/P$, we have  $tH=K$ with
$H=H(A/P), K=K(A/P)$. So $K=tH\subseteq ((K\cap J)+ P)/P\subseteq Z(A/P)$,  and
$H\subseteq t^{-1}Z(A/P)\subseteq Z(A/P)$. Therefore $A/P$ is a field.}

\end{proof}

Let $Ann T= \{ x \in A : xT=Tx=0\}$. Finally we have

\bigskip

\begin{coro}
Let $A$ be a semiprime superalgebra with superinvolution $*$, and let $U$ be a Lie ideal of $[K,K]$. Then
either $[J\cap K,K]\subseteq U$ where $J$ is a nonzero ideal of $A$ or there exists a semiprime ideal
$T$ of $A$ such that $A/Ann T$ satisfies $S(4)$ and $(U+T)/T\subseteq Z_0(A/T)$.
\end{coro}

\begin{proof}[Proof:]
By theorem 4.4 we have that either the first conclusion holds, or, for each prime ideal $P$ of $A$,
either $A/P$ satisfies $S(4)$ or $(U+P)/P \subseteq Z_0(A/P)$. Let $T$ be the intersection of the
prime ideals $P$ of $A$ such that $(U+P)/P\subseteq Z_0(A/P)$. Then
$Ann T$ contains the intersection of those prime ideals $P$ such that $A/P$ satisfies $S(4)$. So we
get that
$A/Ann T$ satisfies $S(4)$, and this proves the result.
\end{proof}

\author{Jes\'us Laliena  \footnote {The author has been
supported by the Spanish Ministerio de Ciencia e Innovaci\'on (MTM 2010-18370-CO4-03).} 
\\{\small Departamento de Matem\'aticas y Computaci\'on}\\
{\small  Universidad de La Rioja}\\
{\small  26004, Logro\~no. Spain}\\
{\small jesus.laliena@dmc.unirioja.es }}

\end{document}